\documentclass[12pt,oneside]{amsart}
\usepackage{latexsym,amssymb,amsmath}
\usepackage{verbatim}
\numberwithin{equation}{section}
\newtheorem{theorem}{Theorem}[section]
\newtheorem{lemma}[theorem]{Lemma}

\newtheorem{corollary}[theorem]{Corollary}

\newtheorem{question}[theorem]{Question}

\theoremstyle{definition}
\newtheorem{definition}[theorem]{Definition}
\newtheorem{remark}[theorem]{Remark}
\newtheorem{example}[theorem]{Example}
\newtheorem{convention}[theorem]{Convention}

\begin{document}


\newcommand{\pnpm}{\mathbb{P}^n \times \mathbb{P}^m}
\newcommand{\al}{\alpha}
\newcommand{\Ix}{I_X}
\newcommand{\Ssz}{\mathcal S_{Z}}
\newcommand{\az}{\alpha_{Z}}
\newcommand{\pr}{\mathbb{P}}
\newcommand{\supp}{\operatorname{Supp}}
\newcommand{\F}{\mathbb{F}}
\newcommand{\popo}{\mathbb{P}^1\times \mathbb{P}^1 }
\newcommand{\prpr}{\mathbb{P}^{n_1}\times \cdots\times \mathbb{P}^{n_r} }
\newcommand{\N}{\mathbb{N}}
\newcommand{\Z}{\mathbb{Z}}
\newcommand{\ui}{\underline{i}}
\newcommand{\pnr}{\mathbb{P}^{n_1}\times \cdots \times \mathbb{P}^{n_r}}
\newcommand{\ua}{\underline{\alpha}}
\newcommand{\un}{\underline{n}}
\newcommand{\uN}{\underline{N}}
\newcommand{\uj}{\underline{j}}
\newcommand{\ut}{\underline{t}}
\newcommand{\ud}{\underline{d}}
\newcommand{\depth}{\operatorname{depth}}
\newcommand{\pd}{\operatorname{pdim}}


\title{Separators of fat points in $\pnpm$}
\thanks{Update: May 24, 2010}

\author{Elena Guardo}
\address{Dipartimento di Matematica e Informatica\\
Viale A. Doria, 6 - 95100 - Catania, Italy}
\email{guardo@dmi.unict.it}
\urladdr{http://www.dmi.unict.it/$\sim$guardo/}

\author{Adam Van Tuyl}
\address{Department of Mathematical Sciences\\
Lakehead University, Thunder Bay, ON, P7B 5E1, Canada}
\email{avantuyl@lakeheadu.ca}
\urladdr{http://flash.lakeheadu.ca/$\sim$avantuyl/}


\begin{abstract}
We introduce
definitions for the separator of a fat point and the degree of a
fat point for a fat point scheme $Z \subseteq \pnpm$, and we
study some of their properties.
\end{abstract}
\maketitle


\section{Introduction}

A separator of a point and its degree
are two tools in the toolbox used to study points in projective
space.  Recall that if $X \subseteq \pr^n$ is a finite set of points,
and $P \in X$, then a {\bf separator} of $P$ is any
homogeneous form $F \in R = k[\pr^n] = k[x_0,\ldots,x_n]$
such that $F(P) \neq 0$, but $F(Q) = 0$ for all $Q \in
X \setminus \{P\}$.  Geometrically, a separator is a hypersurface
that passes through all the points of $X$ except $P$.  The
{\bf degree} of the point $P$, denoted $\deg_X(P)$,
is then the smallest degree of any separator of $P$.  The properties
of separators and their degrees were studied by \cite{ABM,B,BC,GKR,
K,O,So}, among others.

The above cited articles focused predominately on the case of reduced
sets of points in $\pr^n$.  There are two natural ways to
generalize this work.  The first such way is to consider
separators of points in a multiprojective space $\pnr$, as was
the focus of \cite{GV2,GV3,M2}.  The second way is to consider
separators of more arbitrary zero-dimensional schemes in $\pr^n$; the
papers \cite{GMVT,K} take this point of view.  In this paper,
we consider the marriage of these two ideas by studying
separators of non-reduced points (specifically,
fat points) in a multiprojective space.

We restrict ourselves in this paper primarily to the bigraded case of
$\pnpm.$  This restriction has the benefit of simplifying
our notation when compared to the general multigraded situation,
and at the same time, our results are much stronger in this context.
Once we recall the required background in Section 2,
we introduce in Section 3 our definition of a separator for a fat
point in $\pnpm$.  Our approach is similar to that of \cite{GMVT}
in that our definitions are defined in terms of the
bigraded generators of the ideal $I_{Z'}/I_Z$
in $R/I_Z$, where $Z' \subseteq Z \subseteq \pnpm$ are fat point
schemes and $R = k[\pnpm]$.  In Section 4 we
introduce the notion of a good set of minimal separators. Roughly
speaking, a good set of minimal separators allows us to describe a
basis for the vector space $(I_{Z'}/I_Z)_{\ut}$ for all $\ut \in
\N^2$. The main results of this paper are Theorem \ref{pnpmcase}
and Theorem \ref{pnrsepfrombetti}. The first theorem shows that arithmetically
Cohen-Macaulay (ACM) sets of fat points in $\pr^{n} \times
\pr^{m}$ have a good set of minimal separators.  The second
shows that if $R/I_Z$ is Cohen-Macaulay (CM), the degree of
a separator of a fat point is encoded into
the shifts of the last syzygy module of
$I_Z$, generalizing similar results of \cite{ABM,B,GMVT,GV3}.

We wish to point out that although some facts for fat
points in multiprojective spaces follow without any 
difficulty from the methods used in \cite{GMVT}, our main results
require additional development beyond what is done in \cite{GMVT}.
This is the case because  when we move to the case of (non)reduced
points $Z$ in a multiprojective spaces, we are no longer guaranteed that the
associated coordinate ring $R/I_Z$ is CM, and furthermore, even if
$R/I_Z$ is CM, it may not be true that $R/I_{Z'}$ is CM
for subschemes $Z' \subseteq Z$.   The fact that $R/I_Z$ and $R/I_{Z'}$ may
fail to be CM is an obstruction to generalizing some of the proofs
in \cite{GMVT} and at the same time, highlights the importance of
the CM property of zero-dimensional schemes in $\pr^n$.

\noindent {\bf Acknowledgments.}  We thank Martin
Kreuzer and Brian Harbourne for their useful
comments. The second author also thanks
the Universit\`a di Catania for its hospitality while
working on this project. He also received support from GNSAGA and NSERC.

\section{Preliminaries}

We recall the relevant properties of points in $\pnpm$.   The
study of such points was initiated in \cite{GuMaRa,GuMaRa3}; further
properties were developed in \cite{G,GV,GV2,GV3,VT,VT2}.
Throughout, $k$ denotes an algebraically closed field of
characteristic zero.

We shall write $(i_1,i_2)\in \N^2$ as $\ui$. We induce a partial
order on $\N^2$ by setting $(i_1,i_2) \succeq (j_1,j_2)$
if $i_t \geq j_t$ for $t=1,2$. The coordinate ring of the {\bf
biprojective space} $\pnpm$ is the $\N^2$-graded ring
$R=k[x_{0},\ldots,x_{n},y_{0},\ldots,y_{m}]$
where $\deg x_{i} =(1,0)$ and $\deg y_{i}= (0,1)$. A point in this space has the form
\[ P =[a_{0}:\cdots:a_{n}] \times [b_{0}:\cdots:b_{m}] \in \pnpm\]
and its defining ideal $I_P$ in $R$ is a prime ideal of the form
\[I_{P} =(L_{1},\ldots,L_{n},L'_{1},\ldots,L'_{m})\] \noindent where $\deg
L_{i}=(1,0)$ and $\deg L'_i=(0,1)$. When $X=\{P_1,\ldots,P_s\}$ is
a set of $s$ distinct points in $\pnpm$, and $m_1,\ldots,m_s$ are
positive integers, then $I_{Z}=I^{m_1}_{P_1}\cap \cdots \cap
I^{m_s}_{P_s}$ defines a {\bf fat point scheme} (or a {\bf set of fat
points}) which we denote by $Z = m_1P_1 + \cdots + m_sP_s$.  We
call $m_i$ the {\bf multiplicity} of the point $m_i$, and the set
$X$, sometimes denoted by $\supp(Z)$, is the {\bf support} of $Z$.
The {\bf degree}
 of a scheme of fat points $Z = m_1P_1 + \cdots
+ m_sP_s$ is then given by $\deg Z = \sum_{i=1}^s
\binom{m_i+N-1}{m_i-1}$ where $N = n+m$.

The ring $R/I_Z$ has Krull dimension 2, but
$1 \leq \depth R/I_Z \leq 2$ (see \cite{VT}).
When $\dim R/I_Z = 2 = \depth R/I_Z$, we
say $Z$ is
{\bf arithmetically Cohen-Macaulay} (ACM).

We need some results about the nonzero-divisors and longest
regular sequence in $R/I_Z$.

\begin{lemma}\label{nzd}
Let $Z \subseteq \pnpm$ be a set of fat points.
\begin{enumerate}
\item[$(i)$]  There
exist two forms $L_1$ and $L'_1$ such that $\deg L_1 =(1,0)$ and
$\deg L'_1=(0,1)$ and both $\overline{L_1},$ and $\overline{L'_1}$ are
nonzero-divisors on $R/I_Z$.
\item[$(ii)$]  If $Z$ is also ACM, then  there exist
elements $\overline{L}_1,\overline{L'}_1$ in $R/I_Z$ such that
$L_1,L'_1$ give rise to a regular sequence in $R/I_Z$ and $\deg
L_1 =(1,0)$ and $\deg L'_1=(0,1)$.
\end{enumerate}
\end{lemma}

\begin{proof}
Statement $(i)$ is \cite[Lemma 3.3]{VT2} extended to the nonreduced case.  For
$(ii)$,  adapt the proof of \cite[Proposition 3.2]{VT}.
\end{proof}

\begin{remark} \label{regularsequenceremark}
After a change of coordinates, we can assume $L_1 =
x_{0}$ and $L'_1 = y_{0}$ in Lemma \ref{nzd}.
Thus, when $Z$ is ACM, $\{x_{0},y_{0}\}$ (or $\{y_0,x_0\}$)
is the regular sequence on $R/I_Z$.
This also implies
that $x_{0}$ and $y_0$ do not
vanish at any point of $\supp(Z)$.
\end{remark}

We require a lemma about the bigraded resolution of a single
point.  Since $I_P$ is a complete intersection, the proof
is an application of the bigraded Koszul resolution.

\begin{lemma} \label{residealpt}
Let $P \in \pnpm$ be any point.  Then the minimal $\N^2$-graded
free resolution of $R/I_{P}$ has the form
\[0 \rightarrow \mathbb{G}_{N} \rightarrow \mathbb{G}_{N-1}
\rightarrow \cdots \rightarrow \mathbb{G}_1 \rightarrow R
\rightarrow
 R/I_{P} \rightarrow 0\]
where $N = n+m$, $\mathbb{G}_N = R(-n,-m)$ and
$\mathbb{G}_{N-1} = R^{n}(-n+1,-m)\oplus R^m(-n,-m+1).$
\end{lemma}


\section{Defining separators of fat points }

We introduce the definitions of a separator and
its degree for fat points in $\pnpm$.
The main idea is to reduce the multiplicity of
a fat point by one, and then use an ideal that captures the
information about passing from the larger scheme to the smaller one.

The following convention is used to
simplify our hypotheses throughout the paper.

\begin{convention}\label{convention}
Consider the fat point scheme
\[Z := m_1P_1 + \cdots + m_iP_i + \cdots + m_sP_s \subseteq \pnpm,\]
and fix a point $P_i \in \supp(Z)$.  We then let
\[Z':= m_1P_1 + \cdots
+ (m_i-1)P_i + \cdots + m_sP_s,\]
denote the fat point scheme obtained by reducing the multiplicity of
$P_i$ by one.  If $m_i= 1$, then the point $P_i$ does
not appear in the support of $Z'$.
\end{convention}

A separator is now defined in terms
of forms that pass through $Z'$ but not $Z$.

\begin{definition}  Let $Z = m_1P_1 + \cdots + m_iP_i+ \cdots + m_sP_s$ be a set
of fat points in $\pnpm$.  We say that $F$ is a {\bf separator of
the point $P_i$ of multiplicity $m_i$} if $F \in I_{P_i}^{m_i-1}
\setminus I_{P_i}^{m_i}$ and $F \in I_{P_j}^{m_j}$ for all $j \neq
i$.
\end{definition}

When $m_i =1$ for all $i$, then $Z$ is a reduced set of
points, and we recover the definition studied in \cite{GV2,GV3,M2}.
Using the notation of Convention
\ref{convention}, a form $F$ is a separator of the point $P_i$ of
multiplicity $m_i$ if $F \in I_{Z'} \setminus I_Z$. We can
algebraically compare $Z$ and $Z'$ by studying the ideal
$I_{Z'}/I_Z$ in the ring $R/I_Z$.  We recall a simple fact about
this ideal.

\begin{lemma}Let $Z$ and $Z'$ be as in Convention
\ref{convention}. Then there exists $p$ bihomogeneous polynomials
$\{{F}_1,\ldots,{F}_{p}\}$, where each $F_i$ is a separator of
$P_i$ of multiplicity $m_i$, such that in the ring $R/I_Z$, the
ideal $I_{Z'}/I_Z=
\left(\overline{F}_1,\ldots,\overline{F}_{p}\right)$.   Here,
$\overline{F}_i$ denotes the class of $F_i$.
\end{lemma}

\begin{proof} Because $R/I_Z$ is Noetherian, the ideal $I_{Z'}/I_Z$ is
finitely generated.  If $\{\overline{F}_1,\ldots,\overline{F}_p\}$
is a set of generators, then each $F_i \in I_{Z'}\setminus I_Z$.
\end{proof}

\begin{definition} We call  the set of bihomogeneous forms
$\{F_1,\ldots,F_{p}\} \subseteq R$ a {\bf set of minimal separators of
$P_i$ of multiplicity $m_i$} if
\begin{enumerate}
\item[$(a)$] $I_{Z'}/I_Z= \left(\overline{F}_1,\ldots,\overline{F}_{p}\right)$, and
\item[$(b)$] there does not exist a set
$\{G_1,\ldots,G_{q}\}$ with $q<p$ such that $I_{Z'}/I_Z =
\left(\overline{G}_1,\ldots,\overline{G}_{q}\right).$
\end{enumerate}
\end{definition}

\begin{remark}
Our approach is similar to \cite{K} in that we relate
a separator to generators
of an ideal of a smaller scheme modulo an ideal of a larger
scheme.  The focus of \cite{K} was primarily on the
case that $X$ is a zero-dimensional scheme, and $Y \subseteq X$ is a
subscheme with $\deg Y = \deg X - 1$.   Rather than an arbitrary
zero-dimensional scheme, we are interested in
fat point schemes $Z' \subseteq Z$ which normally
have $\deg Z' < \deg Z -1$.
\end{remark}

Our next step is to
develop a fat point analog for the degree of a point.

\begin{theorem}  \label{theorem2}
Let $Z$ and $Z'$ be as in Convention \ref{convention}, and fix a
total ordering $\leq$ of $\N^2$.   Let
$\{F_1,\ldots,F_{p}\}$ and $\{G_1,\ldots,G_{p}\}$ be two
sets of minimal separators of $P_i$ of multiplicity $m_i$. Relabel the
$F_i$'s so that $\deg F_1 \leq \cdots \leq \deg F_{p}$, and
similarly for the $G_i$'s.  Then
\[(\deg F_1,\ldots,\deg F_{p}) = (\deg G_1,\ldots,\deg G_{p}).\]
\end{theorem}

\begin{proof}
Let $W = (I_{Z'}/I_Z)$.  Both
$\left\{\overline{F}_1,\ldots,\overline{F}_p\right\}$
and $\{\overline{G}_1,\ldots,\overline{G}_p\}$ are a minimal set of generators
for this ideal.  The number of generators of degree $\ud$ of $W$
is the dimension of
\[Y = W_{\ud}/(R_{e_1}W_{\ud-e_1} + R_{e_2}W_{\ud-e_2})\]
as a vector space.
Here, $W_{\underline{j}}$ is the vector space of all the forms of
degree $\underline{j}$ in $W$, $R_{e_i}$ denotes the elements of
degree $e_i$ in $R$, and $R_{e_i}W_{\ud-e_i} = \{V_1V_2 ~|~ V_1
\in R_{e_i} ~~\text{and}~~ V_2 \in W_{\ud-e_i}\}$.   The
generators of degree $\ud$ in
$\left\{\overline{F}_1,\ldots,\overline{F}_p\right\}$
and $\{\overline{G}_1,\ldots,\overline{G}_p\}$ therefore
form a basis for $Y$, thus implying that the number
of generators of degree $\ud$ is the same.
\end{proof}

In light of Theorem \ref{theorem2}, we can define the degree of a fat point.

\begin{definition}  Let $\{F_1,\ldots,F_{p}\}$ be any
set of minimal separators of $P_i$ of multiplicity $m_i$, and relabel so
that $\deg F_1 \leq \cdots \leq \deg F_{p}$ with respect to any
total ordering on $\N^2$.  Then the {\bf degree of the minimal
separators of $P_i$ of multiplicity $m_i$} is
\[\deg_Z(P_i) := (\deg F_1,\ldots, \deg
F_{p}) ~~\mbox{with $\deg F_i \in \N^2$}.\]
\end{definition}

We illustrate some of the above ideas with the following two examples.

\begin{example}\label{example1}
Let $Z = mP$ be a single fat point of multiplicity $m \geq
2$ in $\popo$.  We can assume that $I_P =
(x_1,y_1)$, and hence $I_Z = I_P^m$.
Then
\[I_{Z'}/I_Z = I_P^{m-1}/I_P^m =
\left.\left(\overline{M} ~\right|~
M = x_1^{a}y_1^{b} ~~\mbox{with $a + b = m-1$}\right).\]
The generators of $I_P^{m-1}$ are a set of minimal
separators of  $P$ of multiplicity $m$, whence
\[\deg_Z(P) = ((0,m-1),(1,m-2),\ldots,(m-2,1),(m-1,0)).\]
Note that in this case we have $m=|\deg_Z(P)| = \deg Z -\deg Z'.$
The situation where $|\deg_Z(P)| = \deg Z -\deg Z'$ plays
an important role in the next section.
\end{example}

\begin{example}\label{example2}
We consider two fat points  $Z = 2P_1 + 2P_2$ where
$P_{1} = [1:0:0] \times [1:0:0:0]$ and $P_2 = [0:0:1]
\times [0:0:0:1]$ in $\pr^2 \times \pr^3$. Note that $I_Z$
is monomial ideal since $I_{P_1}$ and $I_{P_2}$
are monomial ideals.

Let $Z' = 2P_1 + P_2$.  To find the separators of $P_2$ of
multiplicity $2$, it is enough to determine which generators of
$I_{Z'}$ do not belong to $I_Z$.  Using {\tt CoCoA} \cite{C}, we get
\[
\{x_1x_2, x_1y_3, x_2y_1, x_2y_2,  y_1y_3, y_2y_3, x_0x_2^2,
x_2^2y_0, x_0x_2y_3, x_2y_0y_3, x_0y_3^2, y_0y_3^2\}.\] It then
follows that\[ \deg_{Z}(P_2) =
((0,2),(0,2),(0,3),(1,1),(1,1),(1,1),(1,2),(1,2),(2,0),(2,1),(2,1),(3,0)),\]
where we ordered our tuples with respect to the
lex ordering. Note that $|\deg_Z(P_2)| =
12$, which does not equal $\deg Z - \deg Z' = 5$.  In this case,
$Z$ is not ACM.
\end{example}


\section{Good Separators}

We introduce the notion of a good set of minimal
separators.  Roughly speaking, a minimal set of separators for a
fat point is a good set of separators if the separators
can be used to construct a basis
for the vector space $(I_{Z'}/I_Z)_{\ut}$ for all $\ut \in \N^2$.

Recall that by Remark \ref{regularsequenceremark}
we can assume that none of the points in $\supp(Z)$ lie
on the lines defined by $x_{0}$ and $y_{0}$.  That is,
$x_{0}$ and $y_0$ are nonzero-divisors in the rings $R/I_Z$ and
$R/I_{Z'}$.  So, if $\overline{0} \neq \overline{F} \in
(I_{Z'}/I_Z)$, then $\overline{0} \neq
\overline{x_{0}^{a}y_{0}^{b}F} \in (I_{Z'}/I_Z)$ for any $(a,b)
\in \N^2$. With these observations in hand, we introduce the
following definition.

\begin{definition}
Let $Z$ and $Z'$ be as in Convention \ref{convention}, and let
$\{F_1,\ldots,F_p\}$ be a set of minimal separators of the point
$P_i$ of multiplicity $m_i$.  Let $\deg F_i =
(d_{i1},d_{i2})$. We call $\{F_1,\ldots,F_p\}$ a {\bf good set of
minimal separators}  if
for each $\underline{t} = (t_1,t_2) \in \mathbb {N}^2$ the set
\[\left\{\overline{x_{0}^{t_{1}-d_{11}}y_{0}^{t_{2}-d_{12}}F_1},\ldots,
\overline{x_{0}^{t_{1}-d_{p1}}y_{0}^{t_{2}-d_{p2}}F_p} \right\}\]
is a linearly independent set of elements in
$(I_{Z'}/I_Z)_{\underline t}$, where if $t_{j}-d_{kj}<0$ for some
$k$, then the term $\overline{x_{0}^{t_{1}-d_{k1}}
y_{0}^{t_{2}-d_{k2}}F_{k}}$ is omitted.
\end{definition}

\begin{example}\label{example4}
Consider the points $P_1 = [1:0] \times [1:0]$ and $P_2 =
[1:1] \times [1:1]$ in $\popo$, and set $Z = \{P_1,P_2\}$ and $Z'
= \{P_1\}$.  Thus, $I_Z =
(x_1,y_1) \cap (x_1-x_0,y_1-y_0)$ and $I_{Z'} = (x_1,y_1)$. So
$(I_{Z'}/I_Z) = (\overline{x}_1,\overline{y_1}).$
Now, $\overline{y_0x_1}$ and $\overline{x_0y_1}$ are both
separators of $P_2$ of degree $\ut = (1,1)$ in $(I_{Z'}/I_Z)_{\ut}$.
However, $\overline{y_0x_1} - \overline{x_0y_1} = \overline{0} \in
(I_{Z'}/I_Z)_{\ut}$ because
$y_0x_1-x_0y_1 = y_0(x_1-x_0) - x_0(y_1 - y_0) \in I_Z$, so $\overline{y_0x_1}$ and $\overline{x_0y_1}$ are
not linearly independent.  Thus, $\{\overline{x}_1,\overline{y_1}\}$
is not a good set of minimal separators.
\end{example}

A good set of minimal separators has the following useful
properties.

\begin{theorem} \label{linindep} Let $Z,Z'$ be as in Convention
\ref{convention}. Suppose that $\{F_1,\ldots,F_p\}$  is a
good set of minimal separators of the point $P_i$ of multiplicity $m_i$. Then
\begin{enumerate}
\item [$(i)$] for every $\ut \in \N^2$ a basis for $(I_{Z'}/I_{Z})_{\ut}$ is given by
\[\left\{\overline{x_{0}^{t_{1}-d_{11}}y_{0}^{t_{2}-d_{12}}F_1},\ldots,
\overline{x_{0}^{t_{1}-d_{p1}}y_{0}^{t_{2}-d_{p2}}F_p}\right\}
;\]
\item [$(ii)$] $\dim_k (I_{Z'}/I_Z)_{\ut} =|\{F_i ~|~ \deg F_i \preceq \ut\}|  ~~\mbox{for all $\ut \succeq \underline{0}$; and }$
\item [$(iii)$] $p= \deg Z -\deg Z' =
\binom{m_i+N-1}{m_i -1}-\binom{m_i+N-2}{m_i-2}$, where $N = n+m$.
\end{enumerate}
\end{theorem}

\begin{proof}
Assume that $P = P_i = [1:0:\cdots:0] \times  [1:0:\cdots:0]$ so
that $I_P = (x_{1},\ldots,x_{n},y_{1},\ldots,y_{m})$ .

$(i)$ By definition, the elements
$\left\{\overline{x_0^{t_1-d_{11}}y_0^{t_2-d_{12}}F_1},\ldots,
\overline{x_0^{t_1-d_{p1}}y_0^{t_2-d_{p2}}F_p}\right\}$ form a
linearly independent set in
$(I_{Z'}/I_Z)_{\underline{t}},$ so it suffices to show that
they also span $(I_{Z'}/I_Z)_{\underline{t}}$.  For any
$\overline{H} \in (I_{Z'}/I_Z)_{\underline{t}}$, there must exist
homogeneous forms $G_1,\ldots,G_p$ such that
\[\overline{H} = \overline{G_1F_1 + \cdots + G_pF_p} ~~\mbox{with $\deg G_i = \underline{t}-\deg F_i $}.\]
Rewrite each $G_i$ as $G_i =
c_ix_0^{t_1-d_{i1}}y_0^{t_2-d_{i2}} + G'_i$ with $G'_i \in I_P$.
For each $i=1,\ldots,p$, we have $G'_iF_i \in I_Z$.  To see this,
note that $F_i \in I_{P_j}^{m_j}$ if $P_j \neq P$. On the
other hand, $F_i \in I_{P}^{m_{i}-1}$ and $G'_i \in I_P$, so
$G'_iF_i \in I_P^{m_i}$.  Hence, $G'_iF_i \in I_Z = I_{P_1}^{m_1}
\cap \cdots \cap I_{P}^{m_{i}} \cap \cdots \cap I_{P_s}^{m_s}$.
This implies
\[\overline{H} = \overline{c_1x_0^{t_1-d_{11}}y_0^{t_2-d_{12}}F_1 + \cdots + c_px_0^{t_1-d_{p1}}y_0^{t_2-d_{p2}}F_p},\]
i.e., $\overline{H}$ is in the span of
$\left\{\overline{x_0^{t_1-d_{11}}y_0^{t_2-d_{12}}F_1},\ldots,\overline{
x_0^{t_1-d_{p1}}y_0^{t_2-d_{p2}}F_p},\right\}$.

$(ii)$  This follows directly from $(i)$.

$(iii)$ The second equality can be computed directly from the degree formula.
We prove the first equality.
For all $\ut \in \N^2$ we have a short exact sequence of vector spaces
\begin{equation}\label{ses}
0 \longrightarrow (I_{Z'}/I_Z)_{\ut} \longrightarrow
(R/I_Z)_{\ut} \longrightarrow (R/I_{Z'})_{\ut} \longrightarrow
0.
\end{equation}
Take any $\ut =(t_1,t_2) \gg \underline{0}$, i.e., $t_i \gg
0$ for $i=1,2$.  For any set of fat points $Z$, it is known that
$\dim_k (R/I_Z)_{\ut} = \deg Z$ for $\ut \gg \underline{0}$
(cf. \cite[Proposition 4.4]{SVT}).
Hence, if  $\ut \gg \underline{0}$
\[\dim_k(I_{Z'}/I_Z)_{\ut} = \dim_k(R/I_Z)_{\ut} - \dim_k (R/I_{Z'})_{\ut}
= \deg Z - \deg Z'.\]
But by part $(i)$, for $\ut \gg \underline{0}$,  $\dim_k(I_{Z'}/I_Z)_{\ut} = p$,
so the conclusion follows.
\end{proof}

Recall that the {\bf Hilbert function} of $R/I_Z$ is
the function $H_Z: \N^2 \rightarrow \N$ defined by
\[H_Z(\ut):= \dim_{k} (R/I_Z)_{\ut} = \dim_{k} R_{\ut} - \dim_{k}(I_Z)_{\ut}~~\mbox{for
all $\ut \in \N^2$.}\]
The Hilbert functions of $Z$ and $Z'$ are then linked by
$\deg_Z(P)$ when the minimal separators of $P$ of multiplicity
$m$ are also a good set of minimal separators.  The result
follows directly from Theorem \ref{linindep} $(ii)$
and the short exact sequence ($\ref{ses}$).

\begin{corollary}  \label{hilbertfunction2}
Let $Z$ and $Z'$ be as in Convention \ref{convention}, and suppose
that $\deg_Z(P) = (\ud_1,\ldots,\ud_p)$ and that
$\{F_1,\ldots,F_p\}$ is a good set of minimal separators. Then
\[H_{Z'}({\underline t})  =  H_Z({\underline t}) -
|\{\ud_j ~|~ \ud_j \preceq \ut\}| ~~\mbox{for all $\ut \in
\N^2$.}\]
\end{corollary}

\section{Existence of Good Separators in $\pnpm$}

As Theorem \ref{linindep} suggests, a good set of minimal separators
has some useful properties.  A re-examination of the proof
of \cite[Theorem 3.3]{GMVT} shows that when $Z$ is a set
of fat points in $\pr^n$, then the minimal separators of
the point $P_i$ of multiplicity $m_i$ do form
a good set of minimal separators.  Further examination of this proof
reveals that we need the fact that $Z$ is ACM.
We now show that if $Z \subseteq \pnpm$ is ACM,
then for every point $P \in \supp(Z)$, the set of minimal
separators of $P$ forms a good set of minimal separators.

\begin{theorem}\label{pnpmcase}
Suppose that $Z = m_1P_1 + \cdots + m_sP_s$
is a set of fat points in $\pr^n \times \pr^m$,
and furthermore, suppose that $Z$ is ACM.  If $\{F_1,\ldots,F_p\}$
is a set of minimal separators of the point $P_i$
of multiplicity $m_i$, then $\{F_1,\ldots,F_p\}$
is also a good set of minimal separators.
\end{theorem}

\begin{proof}
After
a change of coordinates, we can assume that $P:= P_i
= [1:0:\cdots:0] \times [1:0:\cdots:0]$
and that $\{x_0,y_0\}$ forms a maximal regular sequence
(see Remark \ref{regularsequenceremark}).

For each $\ut= (t_1,t_2) \in \N^2$,
we wish to show that the set
\[\left\{\overline{x_0^{t_1-d_{11}}y_0^{t_2-d_{12}}F_1},\ldots,
\overline{x_0^{t_1-d_{p1}}y_0^{t_2-d_{p2}}F_p}\right\}\] is a
linearly independent set in $(I_{Z'}/I_Z)_{\ut}$.  We can assume
that $t_1 - d_{j1} \geq 0$ and $t_2-d_{j2} \geq 0$ for all
$j=1,\ldots,p$. If $t_i -d_{ji} < 0$ for some $j$, we simply omit
the term involving $F_j$.

Suppose, for a contradiction, that there exist nonzero constants
$c_1,\ldots,c_p$ such that
\[c_1\overline{x_0^{t_1-d_{11}}y_0^{t_2-d_{12}}F_1}+\cdots
+c_p\overline{x_0^{t_1-d_{p1}}y_0^{t_2-d_{p2}}F_p} = \overline{0} \in (I_{Z'}/I_Z)_{\ut},\]
or equivalently,
\[c_1x_0^{t_1-d_{11}}y_0^{t_2-d_{12}}F_1+\cdots +c_px_0^{t_1-d_{p1}}y_0^{t_2-d_{p2}}F_p \in I_Z.\]
We can reorder the $F_i$'s so that $0 \leq t_1 - d_{11} \leq t_1 - d_{21}
\leq \cdots \leq t_1 - d_{p1}$, and we factor out the largest possible
power of $x_0$, i.e.,
\[x_0^{t_1-d_{11}}(c_1y_0^{t_2-d_{12}}F_1 + \cdots + c_px_0^{d_{11}-d_{p1}}y_0^{t_2-d_{p2}}F_p)
\in I_Z.\]
Because $Z$ is ACM and $\overline{x}_0$ is a nonzero-divisor on $R/I_Z$, we get
\[(c_1y_0^{t_2-d_{12}}F_1 + \cdots + c_ey_0^{t_2-d_{e2}}F_e
+ c_{e+1}x_0^{d_{11}-d_{e+1,1}}y_0^{t_2-d_{e+1,2}}F_{e+1}+ \cdots + c_px_0^{d_{11}-d_{p1}}y_0^{t_2-d_{p2}}F_p)
\in I_Z.\]
Note, in the above expression, we are assuming that
$t_1 - d_{11} = \cdots = t_1 - d_{e1} < t_1 - d_{e+1,1}$.
The above expression thus implies that
\[(c_1y_0^{t_2-d_{12}}F_1 + \cdots + c_ey_0^{t_2-d_{e2}}F_e) \in (I_Z,x_0).\]
We now factor out the largest possible $y_0$ in the above polynomial.
We relabel if necessary so that $t_2 - d_{12} \leq t_2 - d_{i2}$ for $i=2,\ldots,
e$.  So, we get
\[y_0^{t_2-d_{12}}(c_1F_1 + \cdots + c_ey_0^{d_{12}-d_{e2}}F_e) \in (I_Z,x_0).\]

Because $\{x_0,y_0\}$ form a regular sequence on $R/I_Z$, we have
that $\overline{y}_0$ is a nonzero-divisor on $R/(I_Z,x_0)$.  Thus, the previous
expression implies that
\begin{equation}\label{equation1}
(c_1F_1 + \cdots + c_ey_0^{d_{12}-d_{e2}}F_e) \in (I_Z,x_0) \Leftrightarrow
c_1F_1 + \cdots + c_ey_0^{d_{12}-d_{e2}}F_e = H_1 + H_2x_0
\end{equation}
with $H_1 \in I_Z$ and $H_2 \in R$.
Note that if we rearrange the last expression, we get
\[H_2x_0 = (c_1F_1 + \cdots + c_ey_0^{d_{12}-d_{e2}}F_e) - H_1.\]
Since $H_1 \in I_Z \subseteq I_{Z'}$ and $F_1,\ldots,F_e \in I_{Z'}$,
we get $H_2x_0 \in I_{Z'}$.  But $x_0$ is a nonzero-divisor
on $R/I_{Z'}$, so $H_2 \in I_{Z'}$.

So, $H_2 \in I_Z$ or $H_2 \in I_{Z'} \setminus I_Z$ since
$I_{Z'} = (I_{Z'}\setminus I_Z) \cup I_Z$.  However,
if $H_2 \in I_Z$, then this would mean that
\[c_1F_1 \in (I_Z,\hat{F_1},F_2,\ldots,F_p) \Leftrightarrow
(\overline{F}_1,\ldots,\overline{F_p}) = (\overline{F}_2,\ldots,\overline{F}_p)\]
which contradicts the fact that the $F_i$'s are a minimal
set of separators.

So, suppose $H_2 \in I_{Z'}\setminus I_Z$,
or equivalently, $\overline{H}_2 \neq \overline{0}$ in $(I_{Z'}/I_Z)$.
Thus,
\[\overline{H}_2 = \overline{G_1F_1} + \cdots + \overline{G_pF_p}\]
for some $G_1,\ldots,G_p$.  But by degree considerations,
$\deg F_1 \succ \deg H_2$, so $G_1 = 0$.  Hence
\begin{equation}\label{equation2}
H_2 = G_2F_2 + \cdots +G_pF_p + L ~~\mbox{with $L \in I_Z$.}
\end{equation}
If we substitute (\ref{equation2}) into (\ref{equation1}), then we get
\[c_1F_1 + \cdots + c_ey_0^{d_{12}-d_{e2}}F_e = H_1 + [ G_2F_2 + \cdots +G_pF_p + L]x_0\]
which, after rearranging and regrouping, gives
\[c_1F_1 = K + K_2F_2 + \cdots +K_pF_p ~~\mbox{with $K \in I_Z$ and $K_i \in R$.}\]
But this means that $\overline{F}_1 \in (\overline{F}_2,\ldots,\overline{F}_p)
\subseteq R/I_Z$, which again contradicts the fact that the
$F_i$'s are a minimal set of separators.
The conclusion now follows.
\end{proof}

In Example \ref{example2} we noted that $|\deg_Z(P_2)| \neq
\deg Z - \deg Z'$, and that $Z$ was not ACM.  This
can now be deduced from the next corollary.

\begin{corollary}  Let $Z$ and $Z'$ be as in Convention
\ref{convention}.  Suppose that there exists a point
$P$ in $Z$ such that $|\deg_Z(P)| \neq \deg Z - \deg Z'$.
Then $Z$ is not ACM.
\end{corollary}

\begin{proof}
If $Z$ is ACM, then by the previous theorem, every
point has a good set of minimal separators, whence
$|\deg_Z(P)| = \deg Z - \deg Z'$ by Theorem \ref{linindep}.
\end{proof}

\begin{example}
We compute the Hilbert function of $Z = 3P$ in $\popo$. Note
that $Z$ is ACM in $\popo$, so by Theorem \ref{pnpmcase} and
Corollary \ref{hilbertfunction2}, we get
\[H_{3P}(i,j)\!=\! H_{2P}(i,j)\! +\!
|\{\ud_l \!\in \!\deg_{3P}(P) ~|~ \ud_l \preceq (i,j)\}|\] and
$H_{2P} = H_{P}(i,j)\!+\! |\{\ud_l\! \in\! \deg_{2P}(P) ~|~ \ud_l
\preceq (i,j)\}|.$  By Example \ref{example1}, $\deg_{3P}(P) =
((0,2),(1,1),(2,0))$,  and $\deg_{2P}(P) = ((0,1),(1,0))$. Since
$H_P(i,j) = 1$ for all $(i,j) \in \N^2$,
\[H_{3P} =\begin{bmatrix}
1 & 2 & 3 & 3&\cdots & \\
2 & 4 & 5 & 5&\cdots & \\
3 & 5 & 6 & 6&\cdots &  \\
3 & 5 & 6 & 6&\cdots & \\
\vdots & \vdots & \vdots & \vdots & \ddots &
\end{bmatrix}  \]
where position $(i,j)$ of the matrix corresponds to $H_{3P}(i,j)$
(the indexing starts at zero, not one). We can use this procedure
to compute $H_{mP}$ for any fat point $mP \subseteq \popo$.
\end{example}

\begin{remark} In a forthcoming paper \cite{GV4}, the authors
give a formula for the degree of a separator of any fat point of
an ACM fat point scheme $Z \subseteq \popo$ that requires only numerical
information describing $Z$.
\end{remark}


\section{The degree of a separator and the minimal resolution}

In this section, we describe how $\deg_Z(P_i)$ is encoded into the
bigraded minimal free resolution of $I_Z$ under certain
hypotheses.   Our results can be seen as a natural generalization
of the case for reduced points in $\pr^n$ (see \cite{ABM,B}),
reduced points in $\pnr$ (see \cite{GV3}), and fat points in
$\pr^n$ (see \cite{GMVT}).

We start with two technical lemmas
that shall be required for our induction step.

\begin{lemma} \label{idealepunto}
Let $Z$ and $Z'$ be as in Convention \ref{convention}.   If
$\{F_1,\ldots,F_{p}\}$ is a good set of minimal separators of
$P_i$ of multiplicity $m_i$, then
\[(I_Z,F_1,\ldots,F_{j-1}):(F_j) = I_{P_i} ~~\mbox{for $j = 1,\ldots,p$.}\]
\end{lemma}

\begin{proof}
We set $\ud_j := \deg F_j$ for $j = 1, \ldots, p$.

To prove the inclusion $I_{P_i}\subseteq (I_{Z},
F_1,\ldots,F_{j-1}):(F_j)$, note that $F_j \in I_{P_q}^{m_q}$ for
all $q \neq i$, and for $q = i$, $F_jI_{P_i} \subseteq
I_{P_i}^{m_i}$ since $F_j \in I_{P_i}^{m_i-1}$.  Hence $F_jI_{P_i}
\subseteq I_Z \subseteq (I_Z,F_1,\ldots,F_{j-1})$.

Set $P:=P_i$. To prove the other inclusion, we do a change of
coordinates so that $\overline{x}_{0},\overline{y}_{0}$ are
nonzero-divisors on $R/I_Z$ and $P = [1:0:\cdots:0] \times
[1:0:\cdots:0]$. Note that this means that $I_{P} =
({x}_{1},\ldots,{x}_{n},{y}_{1},\ldots,{y}_{m})\,.$ Suppose that
$G \in (I_{Z},F_1,\ldots,F_{j-1}):(F_j)$, i.e., $GF_j \in
(I_Z,F_1,\ldots,F_{j-1})$.  Then there exist forms
$A_1,\ldots,A_{j-1} \in R$ and $A \in I_Z$ such that
\begin{equation}\label{GF}
GF_j = A + A_1F_1 + \cdots + A_{j-1}F_{j-1} \Leftrightarrow GF_j -
(A_1F_1 + \cdots + A_{j-1}F_{j-1}) = A \in I_Z.
\end{equation}
We can take $G,A_1,\ldots,A_{j-1}$ to be bihomogeneous.
Furthermore, if $\deg A  = \underline{d}=(d_1,d_2)$, then $\deg G
= \underline{d}- \ud_j$ and $\deg A_l = \underline{d}- \ud_l$ for
$l=1,\ldots,j-1$.  We also write
\[G = c\underline{x}_0^{\underline{d}-\ud_j} + G'~~
\mbox{and}~~A_l = a_l\underline{x}_0^{\underline{d}-\ud_l} +
A'_l\] where we set $\underline{x}_0^{\underline b}=
x_{0}^{b_1}y_{0}^{b_2}$ with $\underline
 b=(b_1,b_2)$, and $G',A_1',\ldots,A_{j-1}' \in I_{P}$.  Note that
if for some
 $k \in \{1,\ldots,j-1\},$ we have
$\underline{d}-\ud_k \not \succeq \underline 0,$
then the term $A_k F_k$ does not
appear. Our goal is to show that $c=0$, whence $G = G'\in I_{P}$.

It follows that $G'F_{j} \in I_{P}^{m_i}$, and similarly $A'_lF_l
\in I_{P}^{m_i}$ for $l=1,\ldots,j-1$.  Because $F_1,\ldots,F_j
\in I_{P_j}^{m_j}$ for $j \neq i$, we get
\[G'F_j - (A'_1F_1 + \cdots + A'_{j-1}F_{j-1}) \in I_Z.\]
If we subtract this expression from (\ref{GF}), we get
\[c\underline{x}_0^{\underline{d}-\ud_j}F_j - (a_1\underline{x}_0^{\underline{d}-\ud_1}F_1
 + \cdots + a_{j-1}\underline{x}_0^{\underline{d}-\ud_{j-1}}F_{j-1}) \in I_Z.\]
 \noindent
But then in $(I_{Z'}/I_Z)_{\underline d}$ we have
\begin{equation}\label{c}
\overline{c\underline{x}_0^{\underline{d}-\ud_j}F_j -
(a_1\underline{x}_0^{\underline{d}-\ud_1}F_1 + \cdots +
a_{j-1}\underline{x}_0^{\underline{d}-\ud_{j-1}}F_{j-1})} =
\overline{0}.
\end{equation}
Since the separators $F_1,\ldots,F_p$ are a good set
of minimal separators, the elements
\[\left\{\overline{\underline{x}_0^{\underline{d}-\ud_1}F_1}, \ldots,
\overline{\underline{x}_0^{\underline{d}-\ud_{j}}F_{j}}\right\}\]
are linearly independent in
$(I_{Z'}/I_Z)_{\underline d}$. Thus equation (\ref{c}) holds only
if $c = 0$. But this means that $G = G' \in I_{P}$, as desired.
\end{proof}

We need the following result from homological algebra
(see \cite[Exercise 4.1.2]{W});  here, we use $\pd(N)$
to denote the {\bf projective dimension} of an $R$-module $N$.

\begin{lemma} \label{pdimlemma}
Let $0 \rightarrow M \rightarrow M' \rightarrow M''
\rightarrow 0$ be a short exact sequence of $R$-modules.
If $\pd(M'') \neq \pd(M)+1$, then
$\pd(M') = \max\{\pd(M),\pd(M'')\}.$
\end{lemma}

\begin{lemma} \label{middleACM}
Let $Z,Z'$ be as in Convention
\ref{convention}, and suppose that $\{F_1,\ldots,F_p\}$
is a good set of minimal separators of the point
$P_i$ of multiplicity $m_i$.
If $Z'$ is ACM, then
$\pd(R/(I_Z,F_1,\ldots,F_j)) = N = n+m$ for $j=1,\ldots,p$.
\end{lemma}

\begin{proof}
For each $j = 1,\ldots,p$, we have the short exact sequence
\begin{equation}\label{ses3}
0\rightarrow \left(R/(I_Z,F_1,\ldots,F_{j-1}):(F_j)\right)(-\ud_{j})
\xrightarrow{\times F_{j}} R/(I_{Z},F_1,\ldots,F_{j-1}) \rightarrow
R/(I_Z,F_1,\ldots,F_j) \rightarrow 0.
\end{equation}
where $\ud_j = \deg F_j$.
But we know from Lemma \ref{idealepunto} that
$(I_Z,F_1,\ldots,F_{j-1}):(F_j) = I_{P_i}$.  So, the
short exact sequence (\ref{ses3})
becomes
\begin{equation}\label{ses4} 0\longrightarrow
(R/I_{P_i})(-\ud_{j}) \stackrel{\times F_j}{\longrightarrow}
R/(I_{Z},F_1,\ldots,F_{j-1}) \longrightarrow R/(I_{Z},F_1,\ldots,F_j)\longrightarrow 0.
\end{equation}
By Lemma \ref{residealpt}, we have
$\pd (R/I_P) = N$ where
$N=n+m$.  We now do
descending induction on $j$.  When $j = p$, then
$I_{Z'} = (I_Z,F_1,\ldots,F_p)$, and $R/I_{Z'}$ is CM by hypothesis.
Since $\dim R/I_Z = 2$, we have $\pd (R/I_{Z'}) = N$.
For $j = p$, the exact sequence (\ref{ses4}) becomes:
\[0\longrightarrow
(R/I_{P_i})(-\ud_{p}) \stackrel{\times F_{p}}{\longrightarrow}
R/(I_{Z},F_1,\ldots,F_{p-1})
\longrightarrow R/(I_{Z},F_1,\ldots,F_p)\longrightarrow 0.
\]
Because $\pd(R/I_{P_i}) = \pd(R/(I_{Z},F_1,\ldots,F_p)) = N$,
Lemma \ref{pdimlemma} implies
\[\pd R/(I_{Z},F_1,\ldots,F_{p-1}) = \max\{\pd(R/I_{P_i}),
\pd(R/(I_{Z},F_1,\ldots,F_p))\} = N.\]
For $j \leq p-1$, we apply the induction hypothesis to (\ref{ses4})
and again use Lemma \ref{pdimlemma}.
\end{proof}

We come to the main result of this section which states that under
certain hypotheses, the entries of $\deg_Z(P_i)$ are encoded into the minimal
free resolution of $I_Z$.

\begin{theorem}\label{pnrsepfrombetti}
Let $Z,Z'$ be  sets of fat points as in Convention
\ref{convention}.  Suppose that $Z$ is ACM, so that the
minimal $\N^2$-graded free resolution of $R/I_Z$ has
the form
\[ 0 \rightarrow \F_N
{\rightarrow}\cdots\rightarrow \F_1 \rightarrow R \rightarrow
R/I_{Z} \rightarrow 0\]
where $N= n+m$.
If $Z'$ is ACM, then \[\F_N= R(-\ud_1-\underline
N)\oplus\cdots\oplus R(-\ud_p-\underline N)\oplus
\F'_N\] where $\deg_Z(P_i) = (\ud_1,\ldots,\ud_p)$ and
$\underline{N} = (n,m)$.
\end{theorem}

\begin{proof} 
Let $\mathcal{H}_0$ denote the minimal free resolution of $I_Z$
and  let $F_1,\ldots,F_p$ be a set of minimal separators. We order them
with respect to the lexicographical ordering, i.e., $\deg F_1 = \ud_1 \leq \cdots \leq \deg F_p = \ud_p$.
Since $Z$ is ACM, 
the set $F_1,\ldots,F_p$ is also a good set of minimal separators by Theorem \ref{pnpmcase} .
We will add each $F_1,\ldots,F_p$ to $I_Z$
one at a time, and then consider the resolution
of $(I_Z,F_1,\ldots,F_j)$ for $j =1,\ldots,p$.

When $j=1$, we have the short exact sequence
\begin{equation}\label{ses11}
0\rightarrow R/((I_{Z}):(F_{1}))(-\ud_{1}) = (R/I_{P_i})(-\ud_1) \xrightarrow{\times
F_{1}} R/I_{Z} \rightarrow R/(I_{Z},F_{1})\rightarrow 0.
\end{equation}
By Lemma \ref{residealpt}, the resolution of $R/I_{P_i}$ has form
\[
0 \rightarrow \mathbb{G}_{N}=R(-\uN) \rightarrow \mathbb{G}_{N-1}
\rightarrow \cdots \rightarrow \mathbb{G}_1 \rightarrow R
\rightarrow
 R/I_{P_i} \rightarrow 0\] \noindent where $N=n+m.$
Applying the mapping cone construction to (\ref{ses11}) we get a
resolution of $I_1=(I_Z,F_1)$:
\begin{equation}\label{multires}
\mathcal{H}_1:~ 0 \rightarrow R(-\ud_1-\uN) \rightarrow
\mathbb{F}_N \oplus \mathbb{G}_{N-1}(-\ud_1) \rightarrow \cdots \rightarrow \mathbb{F}_1 \oplus R(-\ud_1)
\rightarrow R \rightarrow R/I_{1} \rightarrow 0
\end{equation}
\noindent where $\ud_1=(d_{11},d_{12})$ and $\uN=(n,m).$

The resolution of $I_1$ given in (\ref{multires})
is too long since $\pd (R/I_1) = N$ by Lemma \ref{middleACM}.  Thus,
$R(-\ud_1-\uN)$ must be part of the trivial
complex $\mathcal{T}$, and to obtain a minimal resolution,
the term $R(-\ud_1-\uN)$ must cancel with something in
\[ \mathbb{F}_N \oplus \mathbb{G}_{N-1}(-\ud_1) =
\mathbb{F}_{N}\oplus R^{n}(-d_{11}-n+1,-d_{12}-m)\oplus R^{m}(-d_{11}-n,-d_{12}-m+1).\]

By degree considerations, we cannot cancel the term $R(-\ud_1
-\uN)$ with any of the terms of $ R^{n}(-d_{11}-n+1,-d_{12}-m)\oplus R^{m}(-d_{11}-n,-d_{12}-m+1)$.
Thus, $\mathbb{F}_{N} = \mathbb{F}'_{N} \oplus R(-\ud_1-\uN)$,
i.e., the term $R(-\ud_{1}- \uN)$ must cancel with something in
$\mathbb{F}_{N}$. Note that after we cancel $R(-\ud_1-\uN)$, we
get a resolution of $I_1$ which may or may not be minimal.  We let
\[ \mathcal{H}_1:
0 \rightarrow \mathbb{F}'_{N}\oplus \mathbb{G}_{N-1}(-\ud_1)
\rightarrow \cdots \rightarrow R \rightarrow R/I_{1}
\rightarrow 0\] denote this resolution; we shall require
this resolution at the induction step.

More generally, for our induction step, assume that we have shown
that a resolution of $I_{j-1} = (I_Z,F_1,\ldots,F_{j-1})$ is given
by
\[ \mathcal{H}_{j-1}:
0 \rightarrow \mathbb{F'}_N \oplus  \mathbb{G}_{N-1}(-\ud_1)\oplus
\cdots \oplus \mathbb{G}_{N-1}(-\ud_{j-1})\rightarrow \cdots
\rightarrow R \rightarrow R/I_{j-1} \rightarrow 0\] and that
$\mathbb{F}_N = R(-\ud_1-\uN)\oplus \cdots \oplus R(-\ud_{j-1} -
\uN) \oplus \mathbb{F}'_N$. We have a short exact sequence
\begin{equation}\label{ses2}
0\rightarrow R/((I_{j-1}:(F_{j}))(-\ud_{j})
\xrightarrow{\times F_{j}} R/I_{j-1}\rightarrow R/I_j\rightarrow 0
\end{equation}
where $I_j = (I_Z,F_1,\ldots,F_j)$.

We apply the mapping cone construction to (\ref{ses2})
along with the resolution $\mathcal{H}_{j-1}$ to
make a resolution of $R/I_j$.   Since
$R/((I_{j-1}):(F_{j}))(-\ud_{j}) \cong
R/I_{P_i}(-\ud_{j})$,  the mapping cone produces the
resolution:
\[ \mathcal{K}_{j}: 0 \rightarrow  R(-\ud_{j}-\uN)
\rightarrow \mathbb{F'}_N \oplus \mathbb{G}_{N-1}(-\ud_1)\oplus
\cdots \oplus \mathbb{G}_{N-1}(-\ud_{j})\rightarrow \cdots \rightarrow R
\rightarrow  R/I_j\rightarrow 0.\]
This resolution is too long by Lemma \ref{middleACM}, so
$R(-\ud_{j}-\uN)$ must cancel
with a term in
\[ \mathbb{F'}_N\oplus \mathbb{G}_{N-1}(-\ud_1)\oplus \cdots \oplus
\mathbb{G}_{N-1}(-\ud_{j}).\]
The term $R(-\ud_{j}-\uN)$ cannot cancel with any term in
$\mathbb{G}_{N-1}(-\ud_{j})$ by degree considerations.  So,
suppose that $R(-\ud_{j}-\uN)$ cancels with some term in
\[\mathbb{G}_{N-1}(-\ud_l)=
R^{n}(-d_{l1}-n+1,-d_{l2}-m)\oplus R^{m}(-d_{l1}-n,-d_{l2}-m+1) \]
for some $1\leq l<j.$ Hence, either
\[(-d_{j1}-n,-d_{j2}-m)=
(-d_{l1}-n,-d_{l2}-m+1)\]
\noindent from which we get $d_{j1}=d_{l1}$ , and
$d_{j2}=d_{l2}-1$.   But this is not possible since
we have ordered $\ud_1 \leq \cdots \leq \ud_{p}$ with respect to the
lexicographical ordering. Or \[(-d_{j1}-n,-d_{j2}-m)=
(-d_{l1}-n+1,-d_{l2}-m)\]
\noindent from which we get $d_{j1}=d_{l1}-1$ , and
$d_{j2}=d_{l2}$.   But again this is not possible because of the
ordering of $\ud_1 \leq \cdots \leq \ud_{p}$.

 Hence, the term $R(-\ud_{j}-\uN)$ must
cancel with some term in $\mathbb{F}'_{N}\,.$  Hence, $\mathbb
F'_N = \mathbb{F}''_{N} \oplus R(-\ud_{j}-\uN).$  The result now
follows by induction on $j$.
\end{proof}

As a corollary, we can bound on the rank of the last syzygy
module.

\begin{corollary}\label{rank}
With the hypotheses as in Theorem \ref{pnrsepfrombetti} (i), let
$M = \max\{m_1,\ldots,m_s\}$ and $N=n+m$. Then
\[\operatorname{rk}\mathbb{F}_{N} \geq \binom{M+N-2}{N-1}.\]
\end{corollary}

\begin{proof}
Suppose $P_i$ has multiplicity $M$.  Then by Theorem \ref{linindep} and Theorem
\ref{pnrsepfrombetti},
at least $|\deg_Z(P)| = \deg Z - \deg Z' = \binom{M+N-2}{N-1}$ shifts
appear in $\mathbb{F}_N$.
\end{proof}

\section{Future Directions}

All of the definitions and results in this paper,
except Theorem \ref{pnpmcase}, can be easily
generalized to $\pnr$.   However, the existence
of good sets of minimal separators, when
$r \geq 3$, appears difficult to prove.  We
propose the following question:

\begin{question}\label{conjecture}
Suppose that $Z = m_1P_1 + \cdots + m_sP_s$ is a set of ACM fat
points in $\pnr$. Is it true that the set of minimal separators
for any fat point of $Z$ is a good set of minimal separators?
\end{question}

In the proof of Theorem \ref{pnpmcase}, we used
equation (\ref{equation1}) and the fact that $x_0$ is a
nonzero-divisor to show that $H_2 \in I_{Z'}$, from which we
derive our contradiction. In trying to generalize our proof to the
case $\pnr$ with $r \geq 3$, we end up with an expression similar
to (\ref{equation1}), but involving more
nonzero-divisors.  For example, when $r=3$, (and using the
variables $x_i$, $y_i$ and $z_i$) we can show that
there exists an element of the form $H_1 + H_2x_0 + H_3y_0$ with
$H_1 \in I_Z$ and $H_2,H_3 \in R$, and that this element is some
combination of the separators.  Thus, there is
an element of the form $H_2x_0 + H_3y_0 \in I_{Z'}$, but unlike
the bigraded case, we do not see how to use the fact that  $x_0$
is also a nonzero-divisor.

We end with some evidence for this question.
Question \ref{conjecture} is true for $r=1$ (see proof of \cite[Theorem 3.3]{GMVT})
and $r=2$, as proved in this paper.    Question \ref{conjecture} also
holds if $m_1 = \cdots = m_s = 1$ for any $r \geq 1$.  This result follows
from \cite[Theorem 5.7]{GV2} where it is shown that $|\deg_Z(P)| = 1$
when $Z$ is ACM.  In other words, $(I_{Z'}/I_Z) = (\overline{F})$
is principally generated, and
$\left\{\overline{\underline{x}_0^{\ut-\deg F}F}\right\}$ is a linearly
independent set in $(I_{Z'}/I_Z)_{\ut}$ for all $\ut$.


\end{document}